\documentclass[a4paper,12pt]{amsart}
\usepackage{hyperref,latexsym}
\usepackage{enumerate}

\theoremstyle{plain}
\newtheorem{theorem}{Theorem}
\newtheorem{lemma}{Lemma}
\newtheorem{corollary}{Corollary}

\theoremstyle{definition}

\theoremstyle{remark}

\begin{document}

\title[Cubic and Quartic integrals]
      {Cubic and Quartic integrals \\for geodesic flow on 2-torus
      via system of Hydrodynamic type}

\date{2 October 2011}
\author{Misha Bialy and Andrey E. Mironov}
\address{M. Bialy, School of Mathematical Sciences, Raymond and Beverly Sackler Faculty of Exact Sciences, Tel Aviv University,
Israel} \email{bialy@post.tau.ac.il}
\address{A.E. Mironov, Sobolev Institute of Mathematics and
Laboratory of Geometric Methods in Mathematical Physics, Moscow State University }
\email{mironov@math.nsc.ru}
\thanks{M.B. was supported in part by ISF grant 128/10 and A.M. was
supported by RFBR grant 09-01-00598-a, and by the Leading Scientific
Schools grant NSh-7256.2010.1. It is a pleasure to thank these funds
for the support}


\begin{abstract}In this paper we deal with the classical question
of existence of polynomial in momenta integrals for geodesic flows
on the 2-torus. For the quasi-linear system on the coefficients of
the polynomial integral we study the region (so called elliptic
region) where there are complex-conjugate eigenvalues. We show that
for quartic integrals the other two eigenvalues are real and
necessarily genuinely nonlinear. This observation together with the
property of the system to be Rich (Semi-Hamiltonian) enables us to
classify elliptic regions completely. The case of complex-conjugate
eigenvalues for the system corresponding to the integral of degree 3
is done similarly. These results show that if new integrable
examples exist they could be found only within the region of
Hyperbolicity of the quasi-linear system.

\end{abstract}

\maketitle

\section{Introduction}
\label{sec:intro} Let $\rho$ be a Riemannian metric on the 2-torus
$\mathbb{T}^2=\mathbb{R}^2/\Gamma$, $\rho^t$ denotes the geodesic
flow.  Let $F_n:T^* \mathbb{T}^2$ be a function on the cotangent
bundle which is homogeneous polynomial of degree $n$ with respect to
the fibre (notice that this condition is invariant with respect to
the change of coordinates on the configuration space
$\mathbb{T}^2$). We are looking for such an $F_n$ which is an
integral of motion for the geodesic flow $\rho^t$, i.e.
$F_n\circ\rho^t=F_n$. This question leads immediately to a system of
quasi-linear equations on coefficients of $F_n$ and this is the aim
of the present paper to study it for the degrees $n=3, 4$. Let us
mention that there are classically known examples of the geodesic
flows on the 2-torus which have integrals $F$ of degree one and two.
These examples can be most naturally described with the help of the
conformal coordinates on the covering plane (we refer to the book by
Bolsinov, Fomenko \cite{BF} for the history and discussion of this
classical question):
$$
 (c)
  \ \ \ \ \ \ \ ds^2=\Lambda(q_1,q_2)(dq_1^2+dq_2^2),\
  H=\frac{1}{2\Lambda}(p_1^2+p_2^2).
$$
Here $\Lambda$ is a positive function periodic with respect to the
lattice  $\Gamma$. With these coordinates it can be shown that the
cases of linear and quadratic integral for the geodesic flow
correspond to those conformal factors $\Lambda$ which can be written
as a sum of two functions of one variable:
$$
 \Lambda=f_1(m_1q_1+n_1q_2)+f_2(m_2q_1+n_2q_2) \quad with \quad m_1m_2/n_1n_2=-1,
$$
(where one of the functions must be constant in the case of linear
in momenta integral). Such metrics are called Liouville metrics.  We
shall call a polynomial integral $F_n$ reducible if it can be
written as a polynomial function of the Hamiltonian $H$ and some
other polynomial integral of degree smaller than $n$. In the other
case $F_n$ is called irreducible. Let us mention that there are no
known examples of Riemannian metrics  on the 2-torus having
irreducible integrals of degrees higher than two. This problem is
related also to the so called Birkhoff conjecture on integrable
convex billiards in the plane (see also \cite{B3}, \cite{T}). In
\cite{KD} (see also \cite{KT}) Kozlov and Denisova
 proved that if $\Lambda$ is
trigonometric polynomial then the geodesic flow has no irreducible
polynomial integrals of degree higher than two. Remarkably there do
exist non-trivial examples of geodesic flows on 2-sphere with
integrals which are homogeneous polynomials of degrees 3 and 4.
These examples (see \cite{Sel},\cite{Ki},\cite{BF2},
\cite{DM},\cite{Tsiganov}, \cite{Valent}) were inspired by classical
integrable cases of Goryachev-Chaplygin and Kovalevskaya in rigid
body dynamics. Let us mention also that there are no nontrivial
analytic integrals of the geodesic flows on surfaces of genus higher
than one by Kozlov's theorem  \cite{Koz}.

In what follows we shall work also with  another global coordinate
system on the torus called Semi-geodesic (or equidistant). It is
built with the help of one regular Liouville invariant torus of the
geodesic flow, call it $L$, which projects diffeomorphically to the
configuration space $\mathbb{T}^2$. It follows from \cite{B0} that
such an invariant torus always exists and the Riemannian metric can
be written in the form
$$
 (s)
  \ \ \ \ \ \ \ ds^2=g^2(t,x)dt^2+dx^2,\
  H=\frac{1}{2}\left(\frac{p_1^2}{g^2}+p_2^2\right),
$$
where the family  $\{t=const\}$ is the family of geodesics of the
chosen invariant torus. Interestingly for our approach both
coordinates will play a very important role while each of them (c)
and (s) have their own advantages. In both coordinate systems the
condition of flow-invariance can be reduced to a quasi-linear system
of equations on the coefficients of polynomial $F_n$. For the case
(c) this system gets the form
$$
  A_1(U) U_{q_1} + A_2(U)U_{q_2}=0\eqno{(1)}
$$
while for the case (s) it has a form of evolution equations:
$$
 U_t+A(U)U_x=0.\eqno{(2)}
$$
Here U is a vector function of coefficients and $A_i(U)$ and $A(U)$
are $n\times n$ matrices.

Let us mention here a very important advantage of the second system.
Characteristics of (2) are always transversal to the $x-$direction
but for the system (1) they might rotate. This fact complicates the
analysis along characteristics. Thus in what follows we work with
(2) and then translate the results to (1). (It may happen however
that one can use a recent result \cite{B1} to overcome this
difficulty).

It was proved in our recent paper \cite{BM} that the system (2) is
in fact Rich or Semi-Hamiltonian system. Among other things this
enables following P. Lax and D. Serre (see Serre's book \cite{S}) to
analyze blow up of smooth solutions along characteristics. Such an
analysis was performed for other Rich quasi-linear system (a
reduction of Benney' chain) in \cite{B}. Here we shall apply the
same ideas to the system (2). As usual, for such type of systems one
does not know a-priori that the system is Hyperbolic or it may have
regions with complex eigenvalues. In this paper we shall concentrate
on the cases of degrees 3 and 4. Our main results are for the so
called "elliptic" regions, we shall denote them by $\Omega_e$, these
are regions on the configuration space where the matrix $A(U)$ has
distinct eigenvalues such that two of them are complex-conjugate.
Then it follows that for $n=3$ the third one is obviously real. Also
for $n=4$ the other two eigenvalues must be real as well (see
later). Our main results for the cases  $n=3,4$ say that the
"elliptic" regions are "standard" that is the metric on them is
classically integrable:

\begin{theorem} {\it Let $n=3$, then one has the following alternative:

Either metric is flat in the region $\Omega_e$ or $F_3$ is reducible
on $\Omega_e$ , that is it can be written as combination of $H$ and
$F_1$
$$
 F_3=k_1F_1^3+2 k_2HF_1
$$
for some explicit constants $k_1$, $k_2$.}
\end{theorem}
\begin{corollary} {\it
We have for the conformal model (c):

Either metric $\rho$ is flat on $\Omega_e$ or
$\Lambda=\Lambda(mq_1+nq_2)$ on $\Omega_e$ for some reals $m,n$; If
in addition $\rho$ is known to be real analytic metric on ${\mathbb
T}^2$ then $\Lambda=\Lambda(mq_1+nq_2)$ everywhere on the whole
torus ${\mathbb T}^2$  and the flow $\rho^t$ necessarily has a first
power integral on the whole torus ${\mathbb T}^2$.}
\end{corollary}

\begin{theorem} {\it Let $n=4$, then the following alternative
holds: Either metric $\rho$ is flat on $\Omega_e$ or $F_4$ is
reducible, that is it can be expressed on $\Omega_e$ as
$$
 F_4=k_1F_2^2+2k_2HF_2+4k_3H^2
$$
where $F_2$ is a polynomial of degree 2 which is an integral of the
geodesic flow on $\Omega_e$ and $k_i$ are constants.}
\end{theorem}
\begin{corollary}{\it Either metric $\rho$ is flat on $\Omega_e$ or
the conformal factor $\Lambda(q_1,q_2)$ can be written on $\Omega_e$
in the form
$$
 \Lambda(q_1,q_2)=f(m_1q_1+n_1q_2)+g(m_2q_1+n_2q_2) \quad with \quad \frac{m_1}{n_1}\frac{m_2}{n_2}=-1.
$$
If in addition  $\rho$ is known to be real analytic then $\Lambda$
can be written in such a form for all $q_1,q_2$ on ${\mathbb T}^2$.}
\end{corollary}
There are several main ingredients  in the proof of these results.
The first is the property of the quasi-linear system to be Rich or
Semi-Hamiltonian. This means that it can be written in Riemann
invariants on one hand and in the form of the conservation laws on
the other hand. These facts were proved in our paper \cite{BM}.
Secondly we use strong maximum principle for Riemann invariants
corresponding to complex eigenvalues. Thirdly we were able to show
that for real eigenvalues the condition of genuine nonlinearity is
satisfied.

At the present moment we don't know how to treat the Hyperbolic
regions for the systems (2) and (1). This is mainly because the
possible lack of genuine nonlinearity in the strictly Hyperbolic
regions. We hope to use new analytic and geometric tools to overcome
this difficulty in our further developments.

\section*{Acknowledgements}
We would like to thank  S.P. Novikov, B.A. Dubrovin, O. Mokhov and M. Slemrod
for their encouraging  interest in this paper. We are
grateful to Tel Aviv University for hospitality and excellent
working conditions.

\section {Preparations}
We refer the reader to \cite{BM} for properties of the system
(2)used in this section. We explain first the geometric meaning of
the characteristic polynomial of the matrix $A(U)$ and also prove a
crucial Lemma about it.

The matrix of the system (2) is the following $n\times n$ () matrix
$A$:
$$
 A=  \left(
  \begin{array}{cccccc}
   0 & 0 & \dots &
  0 & 0 & a_1  \\
  a_{n-1} &
 0 & \dots & 0 & 0 & 2a_2-na_0\\0 &
 a_{n-1} & \dots & 0 & 0 & 3a_3-(n-1)a_1\\
 \dots & \dots & \dots & \dots & \dots & \dots \\
 0 &
 0 & \dots & a_{n-1} & 0 & (n-1)a_{n-1}-3a_{n-3}\\
 0 &
 0 & \dots & 0 & a_{n-1} & na_n-2a_{n-2}\\
  \end{array}\right)\eqno{(3)}
,$$ where  it is convenient to write $$ F_n=\sum_{k=0}^n
{a_{k}(t,x)}\frac{p_1^{n-k}}{g^{n-k}}p_2^{k},
$$
is the integral of degree $n$ and the unknown vector function
$U=(a_0,\dots,a_{n-2},a_{n-1})^{T}$ is a column vector of
non-constant coefficients of $F_n$, where one can show that the two
highest coefficients can be normalized to be $$a_{n-1}\equiv g \quad
and \quad a_n\equiv 1.$$

The first interesting property of this system is the fact that its
eigenvalues have very precise geometric meaning: it can be shown
that in order to compute eigenvalues of the system one just have to
find critical points of $F_n$ restricted to the circular fibre of
the energy level. In other words let us differentiate $F_n$ along
the fibre of the energy level:

Let $G_n$ be the homogeneous polynomial which is the derivative of
$F_n$ in the direction of the fibre $\{H=1/2\}\cap T_m^{*}{\mathbb
T}^2$.
$$G_n= L_v (F_n),$$
where the vector field $v$ looks differently for the coordinates (c)
and (s): In case (c)$$
 v=-p_2\frac{\partial}{\partial{p_1
 }}+{p_1}\frac{\partial}{\partial p_2},$$
and in case (s)
$$
 v=-p_2\frac{\partial}{\partial{(p_1/g)
 }}+\frac{p_1}{g}\frac{\partial}{\partial p_2}.$$

Define now usual polynomials $\hat G_n$ and $\hat F_n$ corresponding
to $G_n$ and $F_n$ of the variable $s$ in both models as follows: in
case (c):
$$s=\frac{p_2}{p_1}, \quad
 \hat G_n(s)=\frac{1}{p_1^n}G_n ,\quad \hat
 F_n(s)=\frac{1}{p_1^n}F_n,
$$ and in case (s)
$$
 s=\frac{p_2}{p_1/g}, \quad
 \hat G_n(s)=\frac{g^n}{p_1^n}G_n ,\quad \hat
 F_n(s)=\frac{g^n}{p_1^n}F_n.
$$
 It turns out that with these notations eigenvalues of the matrix $A(U)$ are
 related to the roots $s_i$ of $\hat G_n$(of the model (s)) by
$$
 \lambda_i=gs_i.
$$
Moreover it is remarkable fact that the system (2) can be written in
Riemann invariants: $$
 (r_i)_t+\lambda_i(r_i)_x=0,\  i=1,..,n,
$$
where $r_i$ are just critical values of $F_n$ on the circular
fibres. Therefore for $r_i$ one has the following
$$
 r_i=\frac{p_1^n}{g^n} \hat F_n(s_i)=\frac{\hat F_n(s_i)}{(1+s_i^2)^{n/2}}.
$$
Notice that roots of $\hat G_n$ could be $\pm i$ then the Riemann
invariant $r_i$ would have singularity. However this does not happen
by the following lemma and its corollary. Notice that both
statements are invariantly formulated but for the proof we use the
model (c).

\begin{lemma} {\it Assume $F_n$ is divisible by $H$ for a point on ${\mathbb T}^2$.
Then $F_n$ can be represented
globally as $F_n=HF_{n-2}$. Saying differently if $F_n$ is
irreducible integral then $F_n$ is not divisible by $H$ for any
point.}
\end{lemma}
\begin{proof}  This follows in fact from two identities found by Kolokoltsov in
\cite{kol} for conformal model (c). Write $F_n$ and $H$ in a complex
way as follows $$p=p_1-ip_2, \ F_n=\sum A_ip^{n-i}\bar{p}^i,\
H=\frac{1}{2\Lambda} p\bar{p}.$$ It is shown in \cite{kol} that the
functions $A_0$ and $A_n$ are holomorphic and thus must be constant
on the torus. Therefore if $F_n$ is divisible by $H$ at some point
on the torus then $A_0,A_n$ must vanish at this point and hence
everywhere on the torus $A_0=A_n=0$ . The claim follows.
\end{proof}
\begin{corollary} {\it  For irreducible $F_n$, write the $G_n$ to be derivative of
$F_n$ with respect to the fibre, then $G_n$ is not divisible by $H$
for any point on ${\mathbb T}^2$. In other words $\pm i$ are never
among the roots of $\hat G_n$.}

\end{corollary}
\begin{proof}

 We shall use the coordinates (c) for the proof:
 In complex notations $p,\bar{p}$ one has
$$
  G_n=\frac{i}{2}\left(\bar{p}\frac{\partial F_n}{\partial \bar{p}}-p\frac{\partial F_n}{\partial p}\right).
$$
 If $G_n$ happened to be divisible by $p\bar{p}$ one would have as above $A_0,A_n \equiv 0$. The claim follows.
\end{proof}

\section {Maximum principle for complex Riemann invariants}

Let us recall that we consider for the cases $n=3,4$ those regions
(we shall call them "elliptic" regions) on the configuration space
where all eigenvalues are distinct and two of them are not real. In
such a region the third eigenvalue is obviously real for $n=3$. Also
for $n=4$ the other two eigenvalues must be real as well. This is
because any function on the circle must have maximum and minimum and
as we explained in the previous section eigenvalues  correspond to
critical points of $F_n$ on the fibre. It is important that the
points of maximum and minimum cannot collide up to the boundary of
elliptic region. This is because it is impossible for $F_n$ to be
constant on a fibre. Indeed if $F_n$ is constant on a fibre then the
point of intersection of such a fibre with the torus $L$ would be
necessarily a critical point for $F_n$ ($F_n$ is constant both on
$L$ and on the fibre). But this is impossible since the torus $L$ is
regular. Therefore for $n=4$ the boundary of any elliptic region
 $\partial \Omega_e$ consists of those points where the
complex conjugate roots collide and become real.

We shall use the strong maximum principle for the following
\begin{theorem}{\it Let $s_{1,2}=\alpha \pm i \beta$ be complex
conjugate roots of the polynomial $\hat G_n, n=3,4$. Denote by
$$r_{1,2}=u \pm iv=\frac{\hat F(s)}{(1+s_{1,2}^2)^{n/2}}$$ the
corresponding Riemann invariants. Then $u$ and $v$ must be constants
on $\Omega_e$. Moreover if $\Omega_e$ has a nontrivial boundary then
$v\equiv 0$ on $\Omega_e$.}
\end{theorem}
\begin{proof} Consider first the case $n=4$. Then there is no square
root in the denominator of $r_{1,2}$ and hence by the lemma of
previous section $u,v$ are smooth in the interior and continuous up
to the boundary of $\Omega_e$. In the interior $r=u+iv$ satisfies
the following $$r_t+(g\alpha+ig\beta)r_x=0.$$ Therefore
$$
 (u+iv)_t+(g\alpha+ig\beta)(u+iv)_x=0.
$$
Then denoting $\tilde{\alpha}=g\alpha,\tilde{\beta}=g\beta$ we have
the system:
\[\left\{
\begin{array}{l}
 u_t+\tilde{\alpha}u_x-\tilde{\beta}v_x=0\\
 v_t+\tilde{\beta}u_x+\tilde{\alpha}v_x=0\\
\end{array} \right. \]
Or equivalently,
\[\left\{
\begin{array}{l}
 u_t-\frac{\tilde{\alpha}}{\tilde{\beta}}v_t-
 \frac{\tilde{\alpha}^2+\tilde{\beta}^2}{\tilde{\beta}}v_x=0\\
 u_x+\frac{1}{\tilde{\beta}}v_t+\frac{\tilde{\alpha}}{\tilde{\beta}}v_x=0\\
\end{array} \right. \]
Eliminating $u$ one arrives to the second order equation on the
imaginary part $v$:
$$
 \left(\frac{v_t}{\tilde{\beta}}\right)_t+\left(\frac{\tilde{\alpha}}{\tilde{\beta}}v_x\right)_t+
 \left(\frac{\tilde{\alpha}}{\tilde{\beta}}v_t\right)_x+
 \left(\frac{\tilde{\alpha}^2+\tilde{\beta}^2}{\tilde{\beta}}v_x\right)_x=0.
$$
Its principal part has negative discriminant
 thus the equation is
elliptic. By the strong maximum principle $v$ cannot attain maximum
in the interior point. Therefore $v$ must be constant and moreover
to be zero if there is a non-empty boundary of the elliptic region,
because on the boundary $v$ must vanish. From the system of
equations it follows that $u$ must be a constant as well.

 For the case $n=3$, due to the square root in the formula
one might have not a single-valued function for $r$. However in this
case we shall consider $r^2$ instead, which is also a Riemann
invariant and apply the same argument as above for $r^2$. We have
that $r$ must be a constant.
\end{proof}

\section{Proof of Theorem 1}
In this section we prove Theorem 1. The proof requires the
following:
\begin{theorem} {\it Let $\Omega_e$ be region on ${\mathbb T}^2$ with the property that the polynomial
$$\hat G_3(s)=a_2s^3+(2a_1-3)s^2+(3a_0-2a_2)s-a_1$$ has one real and
two complex conjugate roots. Then it follows that this region is a
strip on the covering plane ${\mathbb R}^2(t,x)$ with the slope
$\lambda$ and on this strip the Riemannian metric admits
1-parametric group of symmetries $g(t,x)=g(\lambda t-x)$ and
therefore there exist a linear integral $F_1=p_1+\lambda p_2.$}
\end{theorem}
\begin{proof}
Firstly we have by Theorem 3 that $r_1,r_2$ are constants on
$\Omega_e.$ On the domain $\Omega_e$ $r_1,r_2,r_3$ can be taken as
coordinates in the space of field variables. $r_1, r_2$ being
constant imply that third one satisfies the equation
$(r_3)_t+\lambda_3(r_3)_x=0,$ where $\lambda_3=s_3g$. Since
$r_1,r_2$ are constant on $\Omega_e$ then $\lambda_3$ depends only
on $r_3$. We have got the simplest quasi-linear equation
$$
 (r_3)_t+\lambda_3(r_3)(r_3)_x=0.
$$
Now the characteristic of this equations are integral curves of the
vector field $$\frac{\partial}{\partial t}+\lambda
_3(r_3(t,x))\frac{\partial}{\partial x}$$ and $r_3$ must be constant
on these curves, hence $\lambda_3$ also remains constant along each
of the curves. This implies that these curves are in fact parallel
straight lines. Take any of these straight lines which passes
through interior point of $\Omega_e$. Then it can not reach the
boundary of $\Omega_e$ and must remain in $\Omega_e$. This follows
from the fact that for interior point $\lambda_1\ne\lambda_2$ and
for the boundary $\lambda_1=\lambda_2$. Moreover along every
characteristic line all $r_i, i=1,2,3$ remain constant then also
$\lambda_i$ because they are defined by coefficients $a_i$ which are
parameterized by $r_1,r_2,r_3$. So we have that $\Omega_e$ is a
strip of the slope $\lambda_3=const.$ Moreover as explained each
$a_i$ has constant values along the characteristic line. In
particular $g=g(x-\lambda_3 t).$ This proves the theorem.
\end{proof}
Now we are in position to complete the proof of main theorem for
$n=3.$

\begin{proof} (Theorem 1). By Theorem 4 $\Omega_e$ in coordinate $(t,x)$ is a strip with the slope
$\lambda$, and $a_i=a_i(x-\lambda t)$, where $a_i$ are functions of
one variable. The coefficients $a_i$ satisfy the quasi-linear system
(2). $$U_t+A(U)U_x=0, U=(a_0,a_1,a_2)^T,$$ which for $n=3$ takes the
form
$$
 A(U)=\left(
 \begin{array}{ccc}
  0&0&a_1\\
  a_2&0&2a_2-3a_0\\
  0&a_2&3-2a_1\\
 \end{array}\right).
$$
Since $U=U(x-\lambda t)$ is in the form  of the simple wave then
$U'$ is $\lambda$-eigenvector of $A(U)$. Writing this fact
explicitly one comes to the following equations:
$$
\begin{array}{l}
 a_1a_2'=\lambda a_0'\\
 a_2a_0'+2a_2a_2'-3a_0a_2'=\lambda a_1'\\
 a_2a_1'+(3-2a_1)a_2'=\lambda a_2'.\\
\end{array}\eqno{(4)}
$$
These differential equations can be solved as follows:

Divide the last equation of (4) by $a_2^3$ to have
$$
 \frac{a_1}{a_2^2}=\frac{3-\lambda}{2}\frac{1}{a_2^2}+c_1,
$$
and so
$$
 \quad a_1=(3-\lambda)/2+c_1a_2^2.\eqno{(5)}
$$

Divide the second equation of the system (4) by $a_2^4$ to have
$$(a_0/a_2^3)^{'}=(1/a_2^2)^{'}+\lambda a_1^{'}/a_2^4.
$$
This expression together with (5) yields
$$
 \frac{a_0}{a_2^3}=\frac{1-\lambda c_1}{a_2^2}+c_2,
$$which means that
$$
 \quad a_0=a_2(1-\lambda c_1)+c_2 a_2^3.\eqno{(6)}
$$
On the other hand substituting $a_1$ from (5) into the first
equation of the system (4) one gets
$$
  \quad \lambda a_0=\frac{c_1}{3}a_2^3+\frac{3-\lambda}{2}a_2+c_3.\eqno{(7)}
$$
Eliminating $a_0$ from the equations (6), (7) one gets certain third
power polynomial on $a_2$ which vanishes. Then there are two
possibilities: either the function $a_2$ is a constant and then the
metric is flat (remember $a_2=g$) or coefficients of this polynomial
must vanish. But this yields the identities:
$$
 c_3=0,\ c_1=\frac{3(\lambda-1)}{2\lambda^2},\ c_2=\frac{\lambda-1}{3\lambda^3}.
$$
Using them one can easily verify the following explicit identity
$$
 F_3=k_1F_1^3+2k_2HF_1,
$$
$$
 k_1=c_2=\frac{\lambda-1}{3\lambda^3},\ k_2=\frac{3-\lambda}{2\lambda}
$$
(where $F_3=\frac{a_0}{a_2^3}p_1^3+\frac{a_1}{a_2^2}p_1^2p_2+p_1p_2^2+p_2^3$).

Let us remark here that the case $\lambda=0$ means $a_i=a_i(x)$ in
particular $g=g(x)$ but then $\rho=g^2(x)dt^2+dx^2$ is obviously a
flat metric.
\end{proof}
\begin{proof} (Corollary 1).
Let us remark first that the fact that the cubic integral can be
explicitly expressed through the first power integral is absolutely
necessary for the proof. This is because the elliptic domain could
be a proper subset of the torus. In such a case it is not
immediately clear why its coefficients should be constants.
Therefore we proceed indirectly as follows.

By the previous theorem we have that the metric $\rho$ posses linear
in momenta integral $F_1$ on the set $\Omega_e$ and moreover $F_3$
can be expressed through $F_1$ and $H$. Write
$F_1=b_0(x,y)p_1+b_1(x,y)p_2.$ Using the identity
$F_3=k_1F_1^3+2k_2HF_1$ and Kolokoltsov constants for $F_3$ we get
$b_0$ and $b_1$ must be constants (obviously at least one of them is
not zero). But then the last equation of quasi-linear system of the
conformal model gives:
$$
 b_0\left(\frac{1}{\Lambda}\right)_{q_1}+b_1\left(\frac{1}{\Lambda}\right)_{q_2}=0,
$$
therefore $\Lambda=\Lambda(b_1q_1-b_0q_2).$ If we know that
$\Lambda$ is of this form on an open subset of ${\mathbb T}^2$ and
if $\Lambda$ is analytic then obviously $\Lambda$ is of this form on
the whole torus ${\mathbb T}^2$. This completes the proof.

\end{proof}

\section{Theorem 2 for the case $\Omega_e \neq {\mathbb T}^2$}
We shall split the proof of Theorem 2 in the two cases:

The first is the case when either $\Omega_e$ has a nontrivial
boundary or it coincides with the whole torus  but the Riemann
invariants $r_{1,2}$ are real at some point (and then everywhere, by
Section 3) on the torus. Our second case (in the next section), when
$\Omega_e$ coincides with the whole torus and $r_{1,2}$ are not real
everywhere on it.

\begin{theorem} {\it Let $F_4:T^*{\mathbb T}^2\rightarrow{\mathbb R}$ be a polynomial of degree 4
such that $\{F_4,H\}=0$. Denote by $\Omega_e$ a domain on ${\mathbb
T}^2$ where polynomial $\hat G_4$ has two complex-conjugate and two
real distinct roots. Then if $\Omega_e\ne {\mathbb T}^2$  or
$\Omega_e= {\mathbb T}^2$ but one knows that ${\rm Im} r_{1,2}\equiv
0$ then $F_4$ can be expressed on $\Omega_e$ as follows
$$
 F_4=k_1F_2^2+2k_2HF_2+4k_3H^2
$$
where $F_2$ is a polynomial of degree 2 which is an integral of the geodesic flow.}
\end{theorem}
\begin{proof}
Denote by $s_{1,2}=\alpha\pm i\beta$ the pair of complex-conjugate
roots of $\hat G_4$, and by $r_{1,2}=u\pm iv$ the corresponding
Riemann invariants. It follows from the Theorem 3 that $v\equiv 0$
and $u$ is a constant on $\Omega_e.$ So we have that $r_{1,2}$ is a
real constant. Denote it by $r$. Then it follows that $F_4-4rH^2$
vanish for $p_2:(p_1/g)=\alpha\pm i\beta$ and thus can be
factorized:
$$
 F_4-4rH^2=KM,\
 K=\left(p_2-\frac{p_1}{g}(\alpha+i\beta)\right)\left(p_2-\frac{p_1}{g}(\alpha-i\beta)\right),
$$
where $M$ is a real polynomial of degree 2.

Moreover, since $(\alpha\pm i\beta)$ are roots of $\hat G_4$  then
$$
 G_4=L_vF=(L_vK)M+K(L_vM)
$$
must be divisible by $K$.
Therefore there are two possibilities:

1. $M$ and $K$ are relatively prime. In this case $L_vK$ must be
divisible by $K$. Notice that for any quadratic polynomial $K$,
$L_vK$ always have real roots. This is because any function on the
circle must have minimum and maximum. Thus the only possibility in
the first case is $L_vK\equiv 0$. But this means $K$ is proportional
to $H$. One may assume  $K=H$ (by taking the coefficient of
proportionality to $M$). So we have got $F_4-4rH^2=HM$, therefore
$M$ must be integral of the geodesic flow degree 2.

2. In this case $M$ is divisible by $K$, i.e. proportional to $K$,
$M=cK$ (where $c(t,x)$ is a function). So we have $F_4-4rH^2=cK^2.$
The function $c$ can be easily computed by equating the coefficients
of $(p_2^4)$ at both sides: $1-r=c.$ Indeed, we have in
Semi-geodesic coordinates $
H=\frac{1}{2}\left(\frac{p_1^2}{g^2}+p_2^2\right)$ and $F_4$ has
coefficient $1$ at $(p_2^4)$.

Thus $c=1-r$ is a constant and we have
$$
 F_4=4rH^2+(1-r)K^2.
$$
Again $K$ must be an integral of degree 2. So in both cases we
proved that $F_4$ is reducible.
This proves Theorem 5.
\end{proof}
By the same method we can prove the following

\begin{corollary}{\it Under the conditions of
Theorem 5 the conformal factor $\Lambda(q_1,q_2)$ can be written on
$\Omega_e$ in the form
$$
 \Lambda(q_1,q_2)=f_1(m_1q_1+n_1q_2)+f_2(m_2q_1+n_2q_2) \quad with \quad m_1m_2/n_1n_2=-1.
$$
Moreover, if $\Lambda$ is known to be real analytic then $\Lambda$
can be written in such a form for all $(q_1,q_2)$ on ${\mathbb
T}^2$.}
\end{corollary}
\begin{proof} We bypass the difficulty exactly as in the previous
corollary. It follows from Theorem 5 that the integral $F_4$ is
reducible on $\Omega_e$, so that there exists an integral $F_2$ on
the domain $\Omega_e$, such that $F_4$ can be expressed as function
of $H$ and $F_2$. Then write $F_2$ in conformal model (c) in the
complex form
 $$F_2=b_0p^2+b_1p\bar{p}+b_2\bar{p}^2.$$
Using the explicit expression
$$
 F_4=k_1F_2^2+2k_2HF_2+4k_3H^2
$$
and Kolokoltsov identities for $F_4$ one can conclude that $b_0,b_2$
must be constants on $\Omega_e$, $b_0=A+iB,b_2=A-iB$ and $b_1$ is a
real function. Then the equations on $b_1$ and $\Lambda$ are the
following:
$$
 (b_1\Lambda)_{q_1}=-2A\Lambda_{q_1}-2B\Lambda_{q_2},\
 (b_1\Lambda)_{q_2}=-2B\Lambda_{q_1}+2A\Lambda_{q_2}.
$$
Eliminating $(b_1\Lambda)$ we get
$$
 B\Lambda_{q_1q_1}-2A\Lambda_{q_1q_2}-B\Lambda_{q_2q_2}=0
$$
It follows from this hyperbolic equation that
$$
 \Lambda(q_1,q_2)=f_1(m_1q_1+n_1q_2)+f_2(m_2q_1+n_2q_2) \quad with \quad m_1m_2/n_1n_2=-1.
$$
which proves the claim.
\end{proof}
\section{Theorem 2 for the case $\Omega_e ={\mathbb T}^2$ and $r_{1,2}$ are not real}
This is the most interesting case which is left. In this case we
shall prove that each one of the real eigenvalues is genuinely
nonlinear in the sense of Lax. This fact, together with the property
of our system to be Rich (or Semi-Hamiltonian), will enable us to
establish blow-up of the derivative $r_x$ unless it vanishes. Such a
proof was first given in \cite{B} for other system.
\begin{theorem} {\it Assume $\Omega_e={\mathbb T}^2,$ and  assume
that for all $(t,x)$ the polynomial $\hat G_4$ has 4 distinct roots,
$2$-complex conjugate $s_{1,2}=\alpha\pm i\beta$ and 2 real
$s_{3,4}$. Assume also that the imaginary part of Riemann invariants
$r_{1,2}$ does not vanish. Then the real eigenvalues
$\lambda_{3,4}=gs_{3,4}$ are necessarily genuinely non-linear and
therefore the corresponding Riemann invariants are constants. In
particular all $a_i$ must be constant, and so the metric is flat.}
\end{theorem}
We proceed as follows. Let us first subtract from $F_4-4H^2$ in
order to make the coefficient of $p_2^4$ vanish, all the other $a_i$
we shall denote by the same letters.
\begin{lemma} {\it Let $\lambda_3=gs_3$ be an eigenvalue of our
quasi-linear system then the derivative
$\frac{\partial \lambda_3}{\partial r_3}$ can be computed
$$
 \frac{\partial \lambda_3}{\partial r_3}=3\frac{(1+s_3^2)^2}{S\hat{G}'_4(s_3)}((a_3+a_1)s_3^2+4a_0s_3-(a_1+a_3))
$$
where
$$
 S=(s_3-s_1)(s_3-s_2)(s_3-s_4),
$$
$$
 \hat{G_4}(s)=a_3s^4+2a_2s^3+3(a_1-a_3)s^2+(4a_0-2a_2)s-a_1.
$$}
\end{lemma}
\begin{proof}
We have
$$
 \frac{\partial\lambda_3}{\partial r_3}=\frac{\partial (s_3a_3)}{\partial r_3}=s_3\frac{\partial a_3}{\partial r_3}+
 a_3\left(\sum_{i=0}^{3}\frac{\partial s_3}{\partial a_i}\frac{\partial a_i}{\partial r_3}\right).
$$
Since $s_3$ is a root of $\hat{G_4}=0$
then
$$
 \frac{\partial s_3}{\partial a_0}=-\frac{4s_3}{\hat{G}_4'(s_3)},\ \frac{\partial s_3}{\partial a_1}=\frac{1-3s_3^2}{\hat{G}_4'(s_3)},\
$$
$$
 \frac{\partial s_3}{\partial a_2}=\frac{2s_3(1-s_3^2)}{\hat{G}_4'(s_3)},\ \frac{\partial s_3}{\partial a_3}=\frac{s_3^2(s_3^2-3)}{\hat{G}_4'(s_3)}.
$$
In order to find $\frac{\partial\lambda_3}{\partial r_3}$ we need to
calculate $\frac{\partial a_i}{\partial r_3}:$
$$
 \left(\frac{\partial a_i}{\partial r_j}\right)=\left(\frac{\partial F(s_i)}{\partial a_j}\right)^{-1}.
$$
From the previous identity we obtain
$$
 \frac{\partial a_0}{\partial r_3}=-\frac{(1+s_3^2)^2s_1s_2s_4}{S},\ \frac{\partial a_1}{\partial r_3}=\frac{(1+s_3^2)^2(s_1s_2+s_1s_4+s_2s_4)}{S},
$$
$$
 \frac{\partial a_2}{\partial r_3}=-\frac{(1+s_3^2)^2(s_1+s_2+s_4)}{S},\ \frac{\partial a_3}{\partial r_3}=\frac{(1+s_3^2)^2}{S}.
$$
 The fact that
$s_i,i=1,\dots,4$ are roots of $\hat{G_4}$ gives us
$$
 s_1+s_2+s_4=-\frac{2a_2}{a_3}-s_3,
$$
$$
 s_1s_2+s_1s_4+s_2s_4=\frac{3(a_1-a_3)}{a_3}-s_3(s_1+s_2+s_4)=$$
$$
 \frac{3(a_1-a_3)}{a_3}+s_3\left(\frac{2a_2}{a_3}+s_3\right),
$$
$$
 s_1s_2s_4=-\frac{a_1}{a_3s_3},
$$
and
$$
 \frac{\partial a_0}{\partial r_3}=\frac{(1+s_3^2)^2a_1}{a_3s_3S},\ \frac{\partial a_1}{\partial r_3}=\frac{(1+s_3^2)^2(3(a_1-a_3)+2a_2s_3+a_3s_3^2)}{a_3S},
$$
$$
 \frac{\partial a_2}{\partial r_3}=\frac{(1+s_3^2)^2(2a_2+a_3s_3)}{a_3S}.
$$
From here we have
$$
 \frac{\partial\lambda_3}{\partial r_3}=-\frac{(1+s_3^2)^2}{S\hat{G}_4'(s_3)}(2a_3s_3^4+4a_2s_3^3+3(a_1-3a_3)s_3^2-4(a_0+a_2)s_3+a_1+3a_3).
$$
Using again that $\hat{G}_4(s_3)=0$ we have
$$
(2a_3s_3^4+4a_2s_3^3+3(a_1-3a_3)s_3^2-4(a_0+a_2)s_3+a_1+3a_3)=$$
$$
 -3((a_3+a_1)s_3^2+4a_0s_3-(a_1+a_3)).
$$
This proves Lemma 2.
\end{proof}
Genuine nonlinearity is proved in the following lemma which is a
general fact about critical values of the polynomials on the circle.
\begin{lemma}
 {\it Let $F_4$ be a polynomial of degree 4,
$F_4=\sum_{i=1}^{3}a_i\left(\frac{p_1}{g}\right)^{n-i}p_2^i$
considered on the circle
$H=\frac{1}{2}\left(\left(\frac{p_1}{g}\right)^2+p_2^2\right)=\frac{1}{2}.$
Denote by $G_4$ the derivative of $F_4$. Let $s_i$ be the roots of
$\hat G_4$ such that $s_{1,2}$ are complex conjugate and $s_{3,4}$
are real distinct. Let $r_i$ denote critical values, with $r_{1,2}$
are complex conjugate (not real) and $r_{3,4}$ real. Then it follows
that for $\lambda_i=gs_i$
$$
  \frac{\partial \lambda_3}{\partial r_3}\ne 0
  \quad and \quad
  \frac{\partial \lambda_4}{\partial r_4}\ne 0.
$$
 }
\end{lemma}
\begin{proof}  In order to prove the lemma denote
$$
 s_{1,2}=\alpha\pm i\beta,\ \beta\ne0.
$$
We have to show that in such a case polynomials $\hat{G}_4(s)$ and
$$
 \gamma=(a_3+a_1)s^2+4a_0s-(a_1+a_3)
$$
are relatively prime (notice that $\gamma$ has only real roots).
This can be done as follows. Divide $\hat G_4$ by $\gamma$ with a
remainder $R$. If $\hat{G}_4$ and $\gamma$ have common root then $R$
has is either of degree zero or is equal to zero. Dividing
explicitly one gets for $R$ the following expression:
$$
 R=$$
$$
 =\frac{2 (a_1^3 + a_1^2 a_3 - a_1 (4 a_0 a_2 + a_3^2)+
 a_3 (8 a_0^2-4 a_0 a_2- a_3^2))(a_1 + a_3 - 4 a_0 s)}{(a_1 +
 a_3)^3}.
$$
Notice that $R$ can be a number in two cases only:

 Case 1. $R\neq0$ but $a_0=0$.

 In this case $\gamma$ has two roots $\pm 1$ and the value of
 $\hat G_4$ is the same for both of them
$$
 \hat G_4(\pm1)=2(a_1-a_3).
$$
 But this means that if one of $\pm 1$ is a common root of $\hat G_4$ and $\gamma$
 then in fact both of them are. But in the first case $R$ is
 not zero, contradiction.

 Case 2. $R=0$.

 In this case $\hat G_4$ is divisible by $\gamma$. Denote $\alpha \pm
 i\beta, \mu, -1/ \mu$ the roots of $ \hat G_4$ where the last two are
 the roots of $\gamma$. With these notations one obtain from the Viete's formula the relation
 between them
$$
 \alpha^2+\beta^2-1= \alpha(\mu - 1/ \mu).
$$
 Notice that if $\alpha=0$ then $\beta=1$ and this case is excluded
 by Lemma 1.
 Moreover the Viete's formulas together with this relation lead to
 the following expressions
$$
 \frac{a_0}{a_3} =\frac{1-(\alpha ^2 + \beta^2)^2}{4\alpha},
$$
$$
 \frac{a_1}{a_3} =\alpha ^2 + \beta^2,
$$
$$
 \frac{a_2}{a_3}=\frac{1-3\alpha^2-\beta^2}{2\alpha}.
$$

 Using these formulas one can substitute them into the value of Riemann
 invariant of the root $\alpha +
 i\beta$ to get by direct calculations that its imaginary part vanishes identically:

$$
 v={\rm Im}\frac{\hat F_4 (\alpha +i\beta)}{(1+(\alpha + i\beta)^2)^2}=$$
$$
 {\rm Im}\frac{a_0+a_1(\alpha +i\beta)+a_2(\alpha +i\beta)^2+a_3(\alpha +i\beta)^3}{(1+(\alpha + i\beta)^2)^2}=0.
$$
 But this is not possible by the assumptions.

 The exceptional cases can be treated
analogously:

Case 3. It could happen that $\gamma$ is identically zero, i.e.
$$
 a_1+a_3=0, \ a_0=0.
$$
In this case $\frac{\hat G_4}{a_3}=s^4+2ks^3-6s^2-2ks-1$, where $k=\frac{a_2}{a_3}$.
One can check that this is impossible because such a polynomial has
4 real roots. One can see this for instance through $\frac{\hat F_4}{(1+s^2)^2}$ which in
this case is
$$
 \frac{\hat F_4}{(1+s^2)^2}=\frac{-a_3s+a_2s^2+a_3s^3}{(1+s^2)^2}.
$$
But the expression in the numerator has zero as a root and two more
real roots. Therefore the derivative $\hat G_4$ must have  4 real
roots. Contradiction.

Case 4. In this case it could happen that $a_1+a_3=0$ and the common
root of $\hat G_4$ and $ \gamma$ is $s=0$. But then $a_1=0$.
Therefore also $a_3=0$. But this is not possible since $a_3=g$ is a
metric coefficient and so is always positive. This finishes the
proof of lemma.

\end{proof}
\begin{proof} (Theorem 6).
It follows from Lemma 3 that real eigenvalues are genuinely
non-linear. It was shown in our previous paper \cite{BM} that the
system (2) is Rich (Semi--Hamiltonian). This property is crucial
because it enables to use Lax analysis of blow up along
characteristics. This method uses Ricatti type equation which
$(r_3)_x$ and $(r_4)_x$ must satisfy. Since one knows that
$\lambda_3,\lambda_4$ are genuinely non-linear (Lemma 3) then
$(r_3)_x$ and $(r_4)_x$ must vanish identically since otherwise
there is a blow up after a finite time for Ricatti equation. We have
already proved $r_1=const, r_2=const.$ Thus all Riemann invariants
are constants and so also $a_i$. In particular $g$ is a constant and
thus Riemann metric is flat.
\end{proof}

\section{Concluding remarks and questions}

1. It would be very interesting to know if our Semi-Hamiltonian
system (2) is in fact Hamiltonian and to find Dubrovin--Novikov
bracket of hydrodynamic type (see \cite{DN}).

2. In this paper we show that in the case $n=3,4$ the system (2)is
standard in the elliptic domain $\Omega_e$. It follows from our main
theorems that in the analytic case we can assume that for $n=3,4$
hyperbolic domain is whole torus ${\mathbb T}^2$. It follows from
\cite{Ts} that in Hyperbolic domain Semi-Hamiltonian systems can be
integrated by Tsarev's generalized hodograph method. It would be
very interesting to apply this method to the system (2).

3. It would be natural to find the blow up mechanism for the
quasi-linear system in the conformal model (c). Technically this is
not obvious because the characteristic fields may rotate on the
torus. However it seems that one can use \cite {B1} to perform
analysis along characteristics of the system (1) directly with no
use of Semi-geodesic coordinates.

\end{document}